%% file: article.tex
\documentclass
{amsart}
\usepackage{amssymb,amscd, amsmath}
\usepackage[all,line,arc,curve,color,frame,pdf]{xy}
\usepackage{tikz}
\usepackage{tikz-cd}
\usetikzlibrary{positioning, trees, snakes}
\usepackage[textsize=tiny]{todonotes}
\usepackage{hyperref}
\usepackage[shortlabels]{enumitem}
\usepackage{float} 
\usepackage[paper=a4paper, margin=3.5cm]{geometry}
\usepackage{cleveref}  
\usepackage{comment}
\usepackage{scalerel}
\usepackage[normalem]{ulem}
\usepackage{listings}
\usepackage{verbatim}

\numberwithin{equation}{section}
\theoremstyle{plain}
\newtheorem{thm}{Theorem}[section]
\newtheorem{cor}[thm]{Corollary}
\newtheorem{lemma}[thm]{Lemma}
\newtheorem{prop}[thm]{Proposition}

\newtheorem{notation}[thm]{Notation}
\newtheorem{remark}[thm]{Remark}

\theoremstyle{remark}
\newtheorem{rem}[thm]{Remark}

\theoremstyle{definition}

\DeclareMathOperator{\Sing}{Sing}

\input{JW-definitions.tex}

\usepackage{listings}
\lstdefinelanguage{Sage}[]{Python}
{morekeywords={False,sage,True},sensitive=true}
\definecolor{dblackcolor}{rgb}{0.0,0.0,0.0}
\definecolor{dbluecolor}{rgb}{0.01,0.02,0.7}
\definecolor{dgreencolor}{rgb}{0.2,0.4,0.0}
\definecolor{dgraycolor}{rgb}{0.30,0.3,0.30}

\begin{document}
\title{Ideals of spaces of degenerate matrices}


\author[Vill]{Julian Vill}
\address{	
	University of Konstanz, Germany, Fachbereich Mathematik und Statistik,
	D-78457 Konstanz, Germany
}
\email{julian.vill@uni-konstanz.de}

\author[Micha{\l}ek]{Mateusz Micha{\l}ek}
\address{
	University of Konstanz, Germany, Fachbereich Mathematik und Statistik, Fach D 197
	D-78457 Konstanz, Germany
}
\email{mateusz.michalek@uni-konstanz.de}

\author[Blomenhofer]{Alexander Taveira Blomenhofer}
\address{University of Konstanz, Germany, Fachbereich Mathematik und Statistik, 
	D-78457 Konstanz, Germany}
\email{alexander.taveira-blomenhofer@uni-konstanz.de}

\begin{abstract} 
	The variety $ \Sing_{n, m} $ consists of all tuples $ X = (X_1,\ldots, X_m) $ of  $ n\times n $ matrices such that every linear combination of $ X_1,\ldots, X_m $ is singular. Equivalently, $X\in\Sing_{n,m}$ if and only if $\det(\lambda_1 X_1 + \ldots + \lambda_m X_m) = 0$ for all $ \lambda_1,\ldots, \lambda_m\in \Q $. Makam and Wigderson \cite{Makam_Wigderson_2019} asked whether the ideal generated by these equations is always \emph{radical}, that is, if any  polynomial identity that is valid on  $ \Sing_{n, m} $ lies in the ideal generated by the polynomials $\det(\lambda_1 X_1 + \ldots + \lambda_m X_m)$. We answer this question in the negative by determining the vanishing ideal of $ \Sing_{2, m} $ for all $ m\in \N $. Our results exhibit that there are additional equations arising from the tensor structure of $ X $. More generally, for any $ n $ and $ m\ge n^2 - n + 1 $, we prove there are equations vanishing on $ \Sing_{n, m} $ that are not in the ideal generated by polynomials of type $\det(\lambda_1 X_1 + \ldots + \lambda_m X_m)$. Our methods are based on classical results about Fano schemes, representation theory and Gr\"obner bases.  
\end{abstract}

\maketitle

\input{intro.tex}
\input{preliminaries.tex}
\input{family1.tex}
\input{large-m.tex}
\input{fano.tex}

\bibliography{bibML}
\bibliographystyle{plain}

\newpage
\section*{Appendix A - Code for \cref{sec:2by2}}
\pagestyle{empty}

\begin{lstlisting}[caption={Sage code verifying that the Gröbner basis constructed in \cref{thm:2by2case-groebner-basis} is indeed a Gröbner basis. }]
# T1 is the matrix containing the first rows of all matrices M1,...,M7
# T2 contains the first columns.
# T3 contains the second rows.
# T4 contains the second columns.
# 
# Mini contains "all" bordered determinants.
# 
# We generate J by all these bordered determinants, I_2 and the products of 2x2 minors of the T_i.
# Now we remove all generators whose leading term is not square-free and check that these form a GB

sage: R.<x1,x2,x3,x4,y1,y2,y3,y4,z1,z2,z3,z4,w1,w2,w3,w4,v1,v2,v3,v4,c1,c2,c3,c4,b1,b2,b3,b4,t1,t2,t3,t4,t5,t6,t7>=PolynomialRing(QQ,order='degrevlex')
sage: M1=matrix([[x1,x3],[x2,x4]])
sage: M2=matrix([[y1,y3],[y2,y4]])
sage: M3=matrix([[z1,z3],[z2,z4]])
sage: M4=matrix([[w1,w3],[w2,w4]])
sage: M5=matrix([[v1,v3],[v2,v4]])
sage: M6=matrix([[c1,c3],[c2,c4]])
sage: M7=matrix([[b1,b3],[b2,b4]])
sage: M=t1*M1+t2*M2+t3*M3+t4*M4+t5*M5+t6*M6+t7*M7
sage: d=det(M)
sage: f1=d.subs(t1=1,t2=0,t3=0,t4=0,t5=0,t6=0,t7=0)
sage: f7=d.subs(t1=0,t2=1,t3=0,t4=0,t5=0,t6=0,t7=0)
sage: f12=d.subs(t1=0,t2=0,t3=1,t4=0,t5=0,t6=0,t7=0)
sage: f16=d.subs(t1=0,t2=0,t3=0,t4=1,t5=0,t6=0,t7=0)
sage: f19=d.subs(t1=0,t2=0,t3=0,t4=0,t5=1,t6=0,t7=0)
sage: f21=d.subs(t1=0,t2=0,t3=0,t4=0,t5=0,t6=1,t7=0)
sage: f2=d.subs(t1=1,t2=1,t3=0,t4=0,t5=0,t6=0,t7=0)-f1-f7
sage: f3=d.subs(t1=1,t2=0,t3=1,t4=0,t5=0,t6=0,t7=0)-f1-f12
sage: f4=d.subs(t1=1,t2=0,t3=0,t4=1,t5=0,t6=0,t7=0)-f1-f16
sage: f5=d.subs(t1=1,t2=0,t3=0,t4=0,t5=1,t6=0,t7=0)-f1-f19
sage: f6=d.subs(t1=1,t2=0,t3=0,t4=0,t5=0,t6=1,t7=0)-f1-f21
sage: f8=d.subs(t1=0,t2=1,t3=1,t4=0,t5=0,t6=0,t7=0)-f7-f12
sage: f9=d.subs(t1=0,t2=1,t3=0,t4=1,t5=0,t6=0,t7=0)-f7-f16
sage: f10=d.subs(t1=0,t2=1,t3=0,t4=0,t5=1,t6=0,t7=0)-f7-f19
sage: f11=d.subs(t1=0,t2=1,t3=0,t4=0,t5=0,t6=1,t7=0)-f7-f21
sage: f13=d.subs(t1=0,t2=0,t3=1,t4=1,t5=0,t6=0,t7=0)-f12-f16
sage: f14=d.subs(t1=0,t2=0,t3=1,t4=0,t5=1,t6=0,t7=0)-f12-f19
sage: f15=d.subs(t1=0,t2=0,t3=1,t4=0,t5=0,t6=1,t7=0)-f12-f21
sage: f17=d.subs(t1=0,t2=0,t3=0,t4=1,t5=1,t6=0,t7=0)-f16-f19
sage: f18=d.subs(t1=0,t2=0,t3=0,t4=1,t5=0,t6=1,t7=0)-f16-f21
sage: f20=d.subs(t1=0,t2=0,t3=0,t4=0,t5=1,t6=1,t7=0)-f19-f21
sage: f22=d.subs(t1=0,t2=0,t3=0,t4=0,t5=0,t6=0,t7=1)
sage: f23=d.subs(t1=0,t2=0,t3=0,t4=0,t5=0,t6=1,t7=1)-f22-f21
sage: f24=d.subs(t1=0,t2=0,t3=0,t4=0,t5=1,t6=0,t7=1)-f22-f19
sage: f25=d.subs(t1=0,t2=0,t3=0,t4=1,t5=0,t6=0,t7=1)-f22-f16
sage: f26=d.subs(t1=0,t2=0,t3=1,t4=0,t5=0,t6=0,t7=1)-f22-f12
sage: f27=d.subs(t1=0,t2=1,t3=0,t4=0,t5=0,t6=0,t7=1)-f22-f7
sage: f28=d.subs(t1=1,t2=0,t3=0,t4=0,t5=0,t6=0,t7=1)-f22-f1
sage: I=ideal(f1,f2,f3,f4,f5,f6,f7,f8,f9,f10,f11,f12,f13,f14,f15,f16,f17,f18,f19,f20,f21,f22,f23,f24,f25,f26,f27,f28)
sage: T=matrix([[x1,x2,x3,x4],[y1,y2,y3,y4],[z1,z2,z3,z4],[w1,w2,w3,w4],[v1,v2,v3,v4],[c1,c2,c3,c4],[b1,b2,b3,b4]])
sage: T1=T[[0..6],[0,1]]
sage: T2=T[[0..6],[0,2]]
sage: T3=T[[0..6],[2,3]]
sage: T4=T[[0..6],[1,3]]
sage: IT1=ideal(T1.minors(2))
sage: IT2=ideal(T2.minors(2))
sage: IT3=ideal(T3.minors(2))
sage: IT4=ideal(T4.minors(2))

sage: Z=[]
sage: Z.append(M1)
sage: Z.append(M2)
sage: Z.append(M3)
sage: Z.append(M4)
sage: Z.append(M5)
sage: Z.append(M6)
sage: Z.append(M7)

sage: mini=[]
sage: for i in [0..6]:   
....:     for j in [0..6]:    
....:         for k in [0..6]:
....:             Q=x1*zero_matrix(3)
....:             Q[[0..1],[0..1]]=Z[i]
....:             Q[[0..1],2]=Z[j][[0..1],0]
....:             Q[2,[0..1]]=Z[k][0,[0..1]]
....:             mini.append(Q.det())
....:             Q[[0..1],2]=Z[j][[0..1],0]
....:             Q[2,[0..1]]=Z[k][1,[0..1]]
....:             mini.append(Q.det())
....:             Q[[0..1],2]=Z[j][[0..1],1]
....:             Q[2,[0..1]]=Z[k][0,[0..1]]
....:             mini.append(Q.det())
....:             Q[[0..1],2]=Z[j][[0..1],1]
....:             Q[2,[0..1]]=Z[k][1,[0..1]]
....:             mini.append(Q.det())
sage: J=I+R.ideal(mini)+(IT1*IT4)+(IT2*IT3)
sage: G=J.gens()
sage: SQF=[]
sage: for i in [0..len(G)-1]:
sage:   if(G[i].lm().is_squarefree()==True):
sage:       SQF.append(G[i])
sage: J1=ideal(SQF)
sage: J1.basis_is_groebner()
\end{lstlisting}
\newpage

\end{document}

%% file: JW-definitions.tex
\usepackage[textsize=tiny]{todonotes}

\usepackage{enumitem}

\newcommand\ignore[1]{}

\newcommand\CC{{\mathbb{C}}}

\def\C{{\mathbb C}}

\def\N{{\mathbb N}}
\def\P{{\mathbb P}}
\def\Q{{\mathbb Q}}

\def\Q{{\mathbb{Q}}}

\def\operatorname#1{\mathop{\rm #1}\nolimits}

\def\Sing{\operatorname{Sing}}

\def\det{\operatorname{det}}

\def\GL{\operatorname{GL}}

\newcommand{\pb}{\ar@{}[dr]|{\text{\pigpenfont J}}}

\makeatletter
\newcommand{\xleftrightarrow}[2][]{\ext@arrow 3359\leftrightarrowfill@{#1}{#2}}
\makeatother
\newcommand{\xdasharrow}[2][->]{
\tikz[baseline=-\the\dimexpr\fontdimen22\textfont2\relax]{
\node[anchor=south,font=\scriptsize, inner ysep=1.5pt,outer xsep=2.2pt](x){#2};
\draw[shorten <=3.4pt,shorten >=3.4pt,dashed,#1](x.south west)--(x.south east);
}}

%% file: intro.tex
\section{Introduction}

\subsection{Motivation} In the late 70's, Valiant \cite{Valiant_1979} proposed an algebraic analogue of the infamous P vs NP question, using a computational model distinct from Turing machines, the model of \emph{arithmetic circuits}. The model of arithmetic circuits captures the natural ways to compute a polynomial function from the basic arithmetic operations, addition and multiplication, via a directed graph encoding how the arithmetic operations are to be nested. Considering such circuits exhibits in many cases non-obvious ways to compute a polynomial that can be much more efficient compared to naively plugging in values into an expanded form. However, since circuits overparameterize polynomials, the problem of polynomial identity testing (PIT), i.e. whether two circuits define the same polynomial, arises naturally. A special case of (PIT) is Determinant Identity Testing (DIT), where a polynomial $ p_A $ in variables $ x_1,\ldots, x_m $, $ m\in \N $ is given by a tuple of matrices $ A = (A_1,\ldots, A_m) \in  (\Q^{n\times n})^{m}$ and the formula
\[ 
	p_A = \det(\sum_{i=1}^m x_i A_i). 
\]
While computing the determinant of a fixed rational matrix is easy, there is no known efficient \emph{deterministic} algorithm to even check if an expression such as $ \det(\sum_{i=1}^m x_i A_i) $ evaluates to the constant zero polynomial. This is curious, since an efficient \emph{probabilistic} algorithm is to simply evaluate such a linear matrix expression in a few random points (which is essentially a consequence of the Schwartz-Zippel lemma \cite{Schwartz_1980}, \cite{Zippel_1993}). Finding such an algorithm would have a major impact towards resolving Valiant's analogue of the P vs NP question \cite{Kabanets_Impagliazzo_2003}. For a detailed introduction to Valiant's classes and circuit complexity, we refer to \cite{Buergisser_Peter_2011}.


\subsection{Algebraic View and Contributions}
Determinant identity testing is the problem to decide membership in the algebraic variety $ \Sing_{n,m} $, i.e.~given a matrix tuple $ A = (A_1,\ldots, A_m) $, determine whether all linear combinations of the $ A_i $ are singular. 
The algebraization of the (DIT) problem was driven e.g.~by Makam and Wigderson \cite{Makam_Wigderson_2019}, who argued that understanding the geometric structure of the variety $ \Sing_{n,m} $ might be a stepping stone on a long climb towards resolving VP vs VNP. 
In this work, we will answer the question of radicality Makam and Wigderson posed in \cite[Problem 12.6]{Makam_Wigderson_2019}, \cite[Problem 4.6]{makam2020symbolic}: Are 
\[ 
\det(\sum_{i=1}^m \lambda_i X^{(i)}) \in \Q[X^{(1)},\ldots, X^{(m)}], \qquad (\lambda \in \Q^m)
\] 
all generators of the vanishing ideal of $ \Sing_{n,m} $?\footnote{Here, $ X^{(1)},\ldots, X^{(m)} $ denote matrices of algebraically independent variables.} While these equations set-theoretically describe the variety $ \Sing_{n,m} $ and while, from a (DIT)-oriented perspective, they form the most natural set of elements in the ideal of $ \Sing_{n,m} $, it turns out that the ideal generated by these polynomials is in general not radical. 
In fact, looking at the problem from the perspective of $ 3 $-tensors makes other identities apparent from rank constraints on different flattenings. Our main theorem is the following.
\begin{thm}
The ideal of the variety $\Sing_{2,m}$ is generated by quadratic polynomials 
\[ 
\det(\sum_{i=1}^m \lambda_i X^{(i)}) \in \Q[X^{(1)},\ldots, X^{(m)}], \qquad (\lambda \in \Q^m)
\] 
if and only if $m\leq 2$. For $m\geq 3$ this ideal is generated by quadrics and cubics.

For any $n$, the ideal of the variety $\Sing_{n,m}$ is not generated by the polynomials \[ 
\det(\sum_{i=1}^m \lambda_i X^{(i)}) \in \Q[X^{(1)},\ldots, X^{(m)}], \qquad (\lambda \in \Q^m)
\] 
if $m>n^2-n$.
\end{thm} 

\subsection{Related work}
By a celebrated result of Kabanets and Impagliazzo \cite{Kabanets_Impagliazzo_2003}, a deterministic polynomial time algorithm for (DIT) would show that either NEXP $\not\subset $ P/poly or the permanent is not computable by polynomially-sized arithmetic circuits. Since the permanent is complete for VNP \cite{Buergisser_Peter_2011}, this would give evidence to a separation of the Valiant classes VP and VNP. The Kabanets-Impagliazzo result is conceptually remarkable in the sense that \emph{constructing} an algorithm for (DIT) would indicate \emph{hardness} of some family in VNP. It has thus motivated extensive studies on the algebraic properties of (DIT).

Perhaps surprisingly, the case of non-commuting variables turned out easier. A line of work (e.g. \cite{Garg_Gurvits_Oliveira_Wigderson_2019}, \cite{Ivanyos_Qiao_Subrahmanyam_2018}) showed the existence of a deterministic polynomial time algorithm for the non-commutative analogue of (DIT). Besides the derandomization of (DIT), other approaches for the resolution of VP vs VNP have been studied, recently e.g.~via sparsity of families of Sum-of-Squares representations \cite{Dutta_Saxena_Thierauf_2021}. Our methods are based on the fundamental article by Chan and Ilten \cite{CI}.

\section*{Acknowledgements}
We thank Visu Makam and Avi Wigderson for inspiring questions and encouraging comments.

%% file: preliminaries.tex
\section{Preliminaries and general notation} 

In what follows, we identify $n\times n$ matrices with elements of the tensor product $\CC^n\otimes\CC^n$. Further, we identify m-tuples $ A = (A_1,\ldots, A_m) $ of $ n\times n $-matrices with $ 3 $-tensors, that is, elements $\sum_{i=1}^m e_i\otimes A_i\in \C^m\otimes\C^n\otimes\C^n $. Let 
\[ 
	\Sing_{n, m} := \{(A_1,\ldots, A_m)\in (\C^n\otimes\C^n)^m \mid \forall \lambda\in \Q^m: \sum_{k=1}^m  \lambda_k A_k \text{ is singular }\} \subseteq \C^m\otimes\C^n\otimes\C^n
\]
be the set of all singular matrix tuples, i.e. all elements such that the image of the corresponding map $\C^m\to \C^n\otimes \C^n$ only contains matrices of rank at most $n-1$. We write $ I_{n, m}$ or $ I(\Sing_{n, m}) $ for the vanishing ideal of $ \Sing_{n, m} $ in the polynomial ring in matrices $ X^{(1)},\ldots, X^{(m)} $ of algebraically independent indeterminates $ (x_{ij}^{(k)})_{i, j\in \{1,\ldots, n\}} $, where $ k\in \{1,\ldots, m\} $. 

The variety $ \Sing_{n, m} $ comes with a natural group action of the group $ \GL_{m} \times \GL_{n} \times \GL_{n} $, which we will call $ G := G_{n, m} $ throughout the paper. $ G $ acts on $ 3$-tensors by ``left-right-back'' multiplication, i.e. in the language of matrix tuples, if $ (U, V, W)\in G $ and $ A = (A_1,\ldots, A_m)\in \Sing_{n, m} $, then
\[ 
	(U, V, W)A = (VA_1W,\ldots, VA_mW)U
\]
Clearly, the determinant of a fixed matrix can only change by a nonzero scalar under left-right multiplication with an element of $\GL_{n} \times \GL_{n} $. On the other hand, linear transformations on the $ \C^m $-mode just correspond to taking different linear combinations, whence $ \Sing_{n, m} $ is invariant under the action of $ G $. We will assume that the reader is familiar with basic representation theory, in particular heighest weight vectors for the general linear group. For a brief introduction we refer to \cite[Chapter 10]{michalek2021invitation} and for a detailed one to \cite{fulton2013representation}.

\subsection{Ideal generation and multiplication maps}

When studying ideal generation via the means of invariant theory, it is crucial to understand what happens if we take the tensor product of irreducible representations: It turns out that the polynomials of degree $ d'+d\in \N_0 $ generated by the degree-$ d $ part of some ideal $ I $ are the image of a sum of tensor products of explicit modules. For $ m\in \N_0 $, tensor products of irreducible $ \GL_m $-modules are understood by the Littlewood-Richardson rule. 

Let $ I $ a homogeneous ideal of a standard graded ring $ R $ (e.g.~a polynomial ring graded by the total degree). For $ d, d'\in \N_0 $, we denote by $ R_{d} $ (resp: $ I_{d} $) the degree-$ d $ component of $ R $ (resp: $ I $). Furthermore, assume that a group $ H $ acts on $ R_d$ and $ I_d$ for any $d$, thus turning both $ I $ and $ R $ into $ H $-modules. Then the multiplication map 
\[ 
I_d \times R_{d'} \to I_{d+d'}, (p, r) \mapsto r\cdot p
\]
induces the following map of $H$-representations
\[ 
M_{I,d,d'}: I_d\otimes R_{d'}\rightarrow I_{d+d'}.
\]
The image of $M_{I,d,d'}$ are those elements in the ideal that are of degree $d+d'$ and lie in the ideal generated by all elements of $ I $ of degree at most $d$ (equivalently: of degree precisely $ d $). In the case  $R =\C[X^{(1)},\dots,X^{(m)}]$ and $ I = I_{n, m} $ we will write $ M_{m, d, d'} $ instead of $ M_{I,d,d'} $ and suppress the dependency on $ n $. 

Now, there are finite collections $ \mathcal{V} $, $ \mathcal{W} $ of irreducible modules such that $ I_d $ and $ R_{d'} $ split into sums
\[ 
I_d = \bigoplus_{V\in \mathcal{V}} V
\]
and
\[ 
R_{d'} = \bigoplus_{W\in \mathcal{W}} W.
\]
The image of $ M_{I,d,d'} $ is then spanned by the images of the modules $ V\otimes W $, where $ (V, W)\in \mathcal{V}\times \mathcal{W} $. In particular, assume $ I = I_{n, m} $ with $ \GL_{m} $ acting on it. Then by the representation theory of the general linear group, the image of $ M_{I,d,d'} $ is spanned by the  images of modules isomorphic to 
\[ 
	S^{\lambda} (\C^m) \otimes S^{\mu} (\C^m)	
\]
where $ \lambda \vdash d $ and $ \mu \vdash d' $. In this case, the Littlewood-Richardson rule describes how such a product of representation decomposes as a sum (cf. \cite[Appendix A.1]{fulton2013representation}).

%% file: family1.tex
\section{The case of $2\times 2$ matrices}\label{sec:2by2}

In this section we restrict our attention to the case $ n = 2 $. We determine generators, and in fact the whole Gr\"obner basis of the vanishing ideal $I(\Sing_{2, m})$. Our main idea is based on the proof of \cite[Theorem 5.8]{michalek2021invitation}.

For $k\in \{1,\ldots, m\}$, let $X_k=(x_{ij}^{(k)})$ be matrices with distinct indeterminate entries.
Let $I=I(\Sing_{2, m})\subset\C[x_{ij}^{(k)}]$ be the homogeneous vanishing ideal of $X$. By \cite[Proposition 5.1]{Makam_Wigderson_2019}, the degree one component $I_1$ is 0, and the degree two component is spanned by the forms $\sum_{i=1}^m \lambda_i \det(X_i)$ for every $\lambda_1,\dots,\lambda_m\in\C$.

Let $J$ be the ideal generated by $I_2$. Then the following holds.

\begin{thm}\label{thm:2by2case-groebner-basis}
If $m\geq 3$ then the ideal $J$ is not radical. The ideal $I=\sqrt{J}$ is generated by $I_2$ and all $3\times 3$ minors of the linear map $\C^m\to \C^2\otimes\C^2$.
\end{thm}
\begin{proof}
We show this by constructing a Gr\"obner basis of an ideal $\tilde I$ that contains $J$ and is contained in $I$. As the leading monomials are all square-free, the ideal $\tilde I$ must be radical. Thus $\tilde I=I$.
As a monomial ordering we consider the degrevlex-ordering, ordering the variables first by matrix, then by column and lastly by row, i.e. $x_{ij}^{(k)}>x_{i'j'}^{(k')}$ if ($k<k'$ ) or ($k=k'$ and $j<j'$) or ($k=k'$ and $j=j'$ and $i<i'$).

In degree two we consider the forms
\begin{equation}
\label{eq}
\det(X_i+X_j)-\det X_i-\det X_j\quad \forall i,j.
\end{equation}
These generate $J_2$: Consider $\lambda_1,\dots,\lambda_m$ as indeterminates and consider the quadratic form $q:=\det(\sum_{i=1}^m \lambda_i X_i)$ in the variables $\lambda_1,\dots,\lambda_m$.
The $\binom{m+1}{2}$ coefficients of $q$ span a subspace of $\C[x_{ij}^{(k)}]$ of dimension at most $\binom{m+1}{2}$. However, all $\binom{m+1}{2}$ forms above are contained in this subspace, and are linearly independent (for $i\neq j$ the form $\det(X_i+X_j)$ is the only one containing any products of the variables of the $i$-th and $j$-th matrix). Hence, these forms form a basis of $J_2$.

In degree three, for every $i,j,k$ we consider all $3\times 3$ minors of the matrix
\[
\begin{pmatrix}
X_i & X_j\\
X_k & 0
\end{pmatrix}.
\]
Let 
\[
T=\begin{pmatrix}
x_{11}^{(1)} & x_{12}^{(1)} & x_{21}^{(1)} & x_{22}^{(1)}\\
x_{11}^{(2)} & x_{12}^{(2)} & x_{21}^{(2)} & x_{22}^{(2)}\\
\vdots & \vdots & \vdots & \vdots\\
x_{11}^{(m)} & x_{12}^{(m)} & x_{21}^{(m)} & x_{22}^{(m)}
\end{pmatrix}
\]
and consider the following $m\times 2$ submatrices
$T_{12},T_{13},T_{24},T_{34}$ where $T_{ij}$ contains the $i$-th and the $j$-th column of $T$.

Let $G_{12},G_{13},G_{24},G_{34}$ be the sets containing the $2\times 2$ minors of the corresponding matrices. Consider the products $G_{12}G_{24},G_{13}G_{34}$.

A Gr\"obner basis $G$ for $\tilde I$ will be given by 
\begin{enumerate}
\item all equations in \cref{eq},
\item all $3\times 3$ minors as above, after removing all forms whose leading term is not square-free,
\item and all elements of the sets $G_{12}G_{24},G_{13}G_{34}$, after removing all forms whose leading term is not square-free.
\end{enumerate}
That all these forms are contained in $I$ is proved in Lemma \ref{lem:contained}.

We prove the theorem relying on a computer algebra system as follows. To prove that $G$ is a Gr\"obner basis, by Buchberger's criterion, we need to show that all $S$-pairs reduce to $0$.
Let $f,g\in G$. By construction $f$ and $g$ each contain the variables of at most four matrices. If $f$ and $g$ both contain variables of four different matrices such that combined there appear variables of eight matrices, then by construction, the variables of $f$ and of $g$ are disjoint. It follows that the $S$-pair of $f$ and $g$ reduces to 0, by Buchberger's second criterion.
If this is not the case, then $f$ and $g$ contain the variables of at most seven different matrices. It therefore suffices to show the statement for $m=2,\dots,7$. We verify this using a computer. Hence, $G$ is a Gr\"obner basis.

Every form of degree two or three in this Gr\"obner basis contains only monomials that are square-free. For forms in $G_{12}G_{24}$ and $G_{13}G_{34}$ this is not true. However, the leading monomial is still square-free. In particular, the leading ideal of $\tilde I$ is radical, and so is $\tilde I$. Thus $\tilde I=I$ and $G$ is in fact a Gr\"obner basis of $I$. We note that the degree four elements of the Gr\"obner basis are generated by quadrics and cubics (which may be checked e.g. using a computer).
\end{proof}

\begin{lemma}
\label{lem:contained}
All forms in $G$ are contained in $I$.
\end{lemma}
\begin{proof}
This is clear for the forms in \cref{eq}.
Modulo $J$, the $3\times 3$ minors we have in $G$ are the same as the $3\times 3$ minors of the map $\C^m\to \C^2\otimes \C^2$. This can be explicitly written down as there are only three different matrices involved.

It thus suffices to show that the map $\C^m\to \C^2\otimes\C^2$ has rank at most 2. We consider the corresponding map between projective spaces
\[
\P\C^m\to \P(\C^2\otimes\C^2).
\]
By assumption the image is contained in the set of rank one matrices which is $\P^1\times \P^1$. This variety only contains three types of subspaces: For every $v,w\in\P^1$ the spaces $v\times\P^1$, $\P^1\times w$ and $v\times w$ are contained. There are no subspaces of (projective) dimension two contained in $\P^1\times\P^1$. Hence the image of $\C^m\to \C^2\otimes\C^2$ has dimension at most two.

Lastly, we need to check the elements of $G_{12}G_{24},G_{13}G_{34}$. Let $A\in X$ be an $m$-tuple of matrices, then the image of the map above has only two possibilities. Either the columns of all matrices in $A$ span a one-dimensional space, or the rows span a one-dimensional space (or both). Hence, either every form in $G_{12}$ vanishes or every form in $G_{24}$ vanishes, and the same holds for $G_{13}$ and $G_{34}$.
\end{proof}

%% file: large-m.tex
\section{There are (essentially) no new equations for $ m>n^2 $}\label{sec:large-m}

For $ m>n^2 $, expressions of the kind $ \sum_{i=1}^m \lambda_i X_i $, ($ \lambda \in \Q^m$) become overparameterized due to the dimension of the matrix space being $ n^2 $. Since membership in $ \Sing_{n, m} $ does set-theoretically only depend on the slices of a given $ 3 $-tensor, one would both hope and expect that the ``complexity'' needed to describe the vanishing ideal $ I(\Sing_{n, m}) $ does not increase beyond $ m=n^2 $. However, we need to be careful with the precise formulation of such a statement. For $ m > n^2 $, clearly $ \Sing_{n, m} $ and $ \Sing_{n, n^2} $ are subvarieties of spaces of different dimensions, with different groups acting on them. 



%
In this short section, we will show that if $ m $ grows beyond $ n^2 $, the ``new'' equations in the vanishing ideal $ I_{n, m} = I(\Sing_{n, m}) $ of $ \Sing_{n, m} $ that make use of the new variables in $ X^{(n^2 + 1)},\ldots, X^{m} $ can be obtained from ``old'' equations in $  I(\Sing_{n, n^2}) $ via the action of $ \GL_{m} $.  To this end, we will make use of the well-known decomposition of the (graded components of the) polynomial ring on $ \C^m\otimes \C^{n^2} $ endowed with the left-right action of $ \GL_m \times \GL_{n^2} $. 


\begin{lemma}\label{lem:left-right-reps}
	Let $ m, q\in \N $. Consider $ \C^m \otimes \C^q$ with the natural left-right action of $ \GL_m \times \GL_q $, turning the ring of polynomial functions 
	\[
		\C[\C^m\otimes \C^q]  \cong \bigoplus_{d\in \N_0} S^d(\C^m\otimes \C^q)
	\]
	into a $ \GL_m \times \GL_q $-module. For $ d\in \N $, we have the following decomposition into irreducible modules
	\[ 
		S^{d}(\C^m\otimes \C^q) \cong \sum_{\lambda \vdash d} S^{\lambda} \C^m \otimes S^{\lambda} \C^q
	\]
	where $ \lambda $ ranges over all nonnegative integer partitions of $ d $. 
\end{lemma}
\begin{proof}
By looking at the characters this is the celebrated Cauchy identity \cite[Appendix A.13]{fulton2013representation}.
\end{proof}

\begin{remark}\label{rem:m-greater-n-squared}
	By definition of the Schur functor, $ S^{\lambda}(\C^q) = 0 $ whenever the partition $ \lambda $ has more than $ q $ parts (i.e.~whenever $ \lambda $ corresponds to a Young tableau with strictly more than $ q $ rows). 
\end{remark}


\begin{notation} 
	For $ m, n, d\in \N$, let $ q:=n^2 $ and fix an identification of the $ \GL_m \times \GL_q $-modules $ S^d(\C^m\otimes \C^q) $ and $ \C[\C^m\otimes \C^q]_{d} $ along with embeddings of the irreducible components of $  S^d(\C^m\otimes \C^q) $ into $ \C[\C^m\otimes \C^q]_{d} $. 
	
	For $ m>n^2 $, we consider the inclusion $ I(\Sing_{n, n^2}) \hookrightarrow  I(\Sing_{n, m}) $. We choose $ n\times n $ matrices of distinct variables $ X^{(1)},\ldots, X^{(m)} $ such that $ \C[\C^m\otimes \C^n \otimes \C^n] \cong \C[X^{(1)},\ldots, X^{(m)}] $ and fix a monomial ordering on the variables such that all entries of $ X^{(1)} $ are greater than all entries of $ X^{(2)} $, which in turn are greater than those of $ X^{(3)} $ and so forth.
\end{notation}

\begin{thm} \label{thm:m-greater-n-squared} For $ m >n^2 $, every irreducible subrepresentation of $ I(\Sing_{n, m}) $ has the highest weight vector in $ I(\Sing_{n, n^2}) $.
\end{thm}

\begin{proof} 
	Let $ V $ be an irreducible $ G $-submodule of $ I(\Sing_{n, m})_{d} $. Then $ V\hookrightarrow S^{\lambda} (\C^m) \otimes S^{\lambda}(\C^{n^2}) $ by the above considerations for some partition $ \lambda \vdash d$ with at most $ n^2 $ parts. The highest weight vector $f_\lambda\in V$, by definition has weight given by $\lambda$ \cite[Chapter 10.2]{michalek2021invitation}. In particular, it is a linear combination of monomials that use only the variables in  $ X^{(1)},\ldots, X^{(n^2)} $. That is $f_\lambda\in I(\Sing_{n, n^2})$.	
\end{proof}
\begin{cor}\label{cor:n^2}
For fixed $n$, all ideals $I_{n,m}$ are generated in degree at most $d$ if and only $I_{n,n^2}$ is generated in degree $d$.
\end{cor}
\begin{proof}
Let $R_{n,m}:=\CC[X^{(1)},\dots,X^{(m)}]$. 
The multiplication map induces a map of $\GL_m$ representations:
\[M_{m,d,d'}:(I_{n,m})_d\otimes (R_{n,m})_{d'}\rightarrow (I_{n,m})_{d+d'}.\]
The image of $M_{m,d,d'}$ are those elements in the ideal that are of degree $d+d'$ and are generated by elements of degree $d$. In particular, our assumptions tell us that $M_{m,d,d'}$ is surjective for $m=n^2$ and any $d'$. 

First we consider the more basic case $m<n^2$. 
Let us break $(I_{n,n^2})_d$ (resp.~$(R_{n,n^2})_{d'}$, resp.~$(I_{n,n^2})_{d+d'}$) into a direct sum of representations $I_1\oplus I_2$ (resp.~$R_1\oplus R_2$, resp.~$J_1\oplus J_2$) , where in $I_1$ (resp.~$R_1$, resp.~$J_1$) are all isotypic components corresponding to Young diagrams with at most $m$ rows and in $I_2$ (resp.~$R_2$, resp.~$J_2$) are those with more than $m$ rows.  We note that the highest weight vectors in $I_1$ (resp.~$R_1$, resp.~$J_1$) are exactly the highest weight vectors in $(I_{n,m} )_d$ (resp.~$(R_{n,m})_{d'}$, resp.~$(I_{n,m})_{d+d'}$). Further, all Young diagrams corresponding to representations in  $M_{n^2,d,d'}\left(I_2\otimes (R_{n,n^2})_{d'}\right)$ and $M_{n^2,d,d'}\left((I_{n,n^2})_d\otimes (R_2)_{d'}\right)$ must have at least $m+1$ rows. Hence, the image of $I_1\otimes R_1$ must contain $J_1$. For each Young diagram with at most $m$ rows, the multiplicity of the corresponding isotypic component in the representation $I_1\otimes R_1$ is the same as in $(I_{n,m})_d\otimes (R_{n,m})_{d'}$ (but the representations are not the same, as these are representations of different groups). Similarly $J_1$ and  $(I_{n,m})_{d+d'}$ are represented by exactly the same Young diagrams. For contradiction assume that $M_{m,d,d'}$ is not surjective. This would mean that there is a (highest weight) vector in the kernel of $M_{m,d,d'}$ that, after mapping to $I_1\otimes R_1$ is not in the kernel of $M_{n^2,d,d'}$. This is not possible, as multiplication of polynomials gives the same result, no matter in a ring of how many variables we regard the polynomials.

With these tools the more important case $m>n^2$ follows easily. Indeed, to show that $M_{m,d,d'}$ is surjective it is enough to prove that every highest weight vector is in the image. But any highest weight vector must be in fact an element of $(I_{n,n^2})_{d+d'}$. In particular, it must be in the image of the smaller space 
$(I_{n,n^2})_d\otimes (R_{n,n^2})_{d'}\subset (I_{n,m})_d\otimes (R_{n,m})_{d'}$.
\end{proof}
\begin{rem}
Corollary \ref{cor:n^2} gives another easy proof of Theorem \ref{thm:2by2case-groebner-basis}, as it is enough to check it for $m=4$. We note that further improvements are possible, due to the fact that $\Sing_{n,m}$ is contained in a subspace variety. This topic will be explained in the forthcoming section.
\end{rem}

%% file: fano.tex
\section{Connections to the Fano scheme}

Let $I_{n,m}\subset\C[x_{ij}^{(k)}]=\C[\C^m\otimes\C^n\otimes\C^n]$ be the homogeneous vanishing ideal of $\Sing_{n,m}$. As shown in the last section, there are no 'new' generators for $m>n^2$. This followed solely by considering the polynomial ring $\C[x_{ij}^{(k)}]$ and the action of $\GL_m\times\GL_n\times\GL_n$ on it.
In particular, we did not make use of the fact that we are only interested in generators of $I_{n,m}$.

We do this now using Fano schemes. Let $D_n\subset\P^{n^2-1}$ be the subvariety consisting of all singular matrices.
Let $k$ be a positive integer. We denote by $\mathbf{F}_k(D_n)$ the Fano scheme of $D_n$, which parametrizes the $k$-dimensional planes in $\P^{n^2-1}$ that are subvarieties of $D_n$.

Let $X\in\Sing_{m,n}\subset\C^m\otimes\C^n\otimes\C^n$. Let $L$ be the image of the flattening $\C^m\to\C^n\otimes\C^n$. Since $X\in\Sing_{m,n}$ every point in $L$ is singular which means $L\subset D_n$. By definition we hence see that $L\in\mathbf{F}_{\dim(L)-1}(D_n)$. This Fano scheme has been extensively studied for example by Chan and Ilten in \cite{CI}. In the following we will make use of these results. 

We show how to get new generators of $I_{n,m}$ using equations coming from this Fano scheme.

\begin{thm}
\label{thm:fano_empty}
Let $n>1$ and $m\ge n^2-n+1$. The ideal $I_{n,m}$ is not generated in degree $n$.
\end{thm}
\begin{proof}
Let $X=(X_1,\dots,X_m)\in\Sing_{n,m}\subset\C^m\otimes\C^n\otimes\C^n$. Then $X$ defines the linear subspace $L=\mathrm{span}(X_1,\dots,X_m)\subset (\C^n\otimes\C^n)^m$ which is contained in $D_n$.
Let $k$ be the dimension of $L$, then $L$ defines a point in $\mathbf{F}_{k-1}(D_n)$.

By \cite[Proposition 2.6.]{CI} $\mathbf{F}_{k-1}(D_n)$ is empty if and only if $k> n(n-1)$. Since $m\ge n^2-n+1$, we have $m>k$ and the matrices $X_1,\dots,X_m$ are linearly dependent.

Equivalently the rank of the flattening $\C^m\to\C^n\otimes\C^n$ is at most $n^2-n$, which means that all $(n^2-n+1)\times(n^2-n+1)$-minors of this map vanish. More explicitly those are the minors of the matrix
\[
\begin{pmatrix}
x_{11}^{(1)} 	& x_{12}^{(1)} 	& \dots 	& x_{nn}^{(1)}\\
\vdots 			& \vdots				&		& \vdots\\
x_{11}^{(m)} & x_{12}^{(m)} & \dots & x_{nn}^{(m)}
\end{pmatrix}
\]
if coordinates are chosen such that $X_l=(x_{ij}^{(l)})_{ij}$ for $l=1,\dots,m$.
By construction every subspace $L$ coming from $X\in \Sing_{n,m}$ satisfies these equations and therefore they are contained in the ideal $I_{n,m}$.

To show that these polynomials of $I_{n,m}$ are not generated by the polynomials in degree up to $n$ we consider the group action of $G=\GL_m\otimes\GL_n\otimes\GL_n$ on $I_{m,n}$.
The equations we found come from the $(n^2-n+1)$-st exterior power of the flattening map $\CC^m\rightarrow (\CC^n\otimes\CC^n)$. Thus correspond to $\bigwedge^{n^2-n+1}\CC^m\otimes \bigwedge^{n^2-n+1}(\CC^n\otimes\CC^n)$. As a $G$ representation this is a union of irreducible representations (with multiplicity one) indexed by triples of Young diagrams $(1,\dots,1)=1^{n^2-n+1},\lambda,\lambda^T$, where $\lambda$ is (any) partition of $n^2-n+1$ fitting in an $n\times n$ square and $\lambda^T$ is its transpose.

By \cite[Corollary B.6 and Proposition 5.1]{Makam_Wigderson_2019} we know that $I_{n,m}$ up to degree $n$ consists just of one irreducible representation indexed by the triple of Young diagrams: $(n),1^n,1^n$.
The image of the tensor product by the multiplication map: $(I_{n,m})_d \otimes \CC[X^{(1)},\dots,X^{(m)}]_{d'}\rightarrow (I_{n,m})_{d+d'}$ is the part of the ideal in the degree $d+d'$ generated by elements in degree $d$. Taking $d=n$, the domain of this map is a sum of irreducible representations indexed by triples of Young diagrams $\mu_1,\mu_2,\mu_3$, where, by the Littlewood-Richardson rule, $\mu_1$ must contain $(n)$ (and $\mu_2,\mu_3$ must contain $1^n$, however we do not need this). In particular, the equations we presented cannot belong to the image of the map, as $1^{n^2-n+1}$ does not contain $(n)$.
\end{proof}

\begin{rem}
In the case $n=2$, the new equations are exactly the cubics in \cref{thm:2by2case-groebner-basis}. In the case $n=2$, those are all generators of $I_{2,m}$ as we have seen. However, for larger $n$ this is most likely not true.
\end{rem}

We recall that by \cite[Proposition 5.1]{Makam_Wigderson_2019} for any $d<n$ the degree $d$ component $(I_{n,m})_d$ is $\{0\}$ and for $d=n$ we have $(I_{n,m})_n=\mathrm{span}\{\det(\sum c_iX_i)\colon c_i\in\C\}$.

\begin{cor}
The ideal generated by the degree $n$ component $(I_{n,m})_n=\mathrm{span}\{\det(\sum c_iX_i)\colon c_i\in\C\}$ of $I_{n,m}$ is never radical if $m\ge n^2-n+1$.
\end{cor}

\begin{rem}
We note that new polynomials in $I_{m,n}$ from \cref{thm:fano_empty} were a result of rank restrictions on the map $\C^m\to\C^n\otimes\C^n$.
However, there are also two more canonical flattenings of tensors in $\C^m\otimes\C^n\otimes\C^n$, namely $\C^n\to\C^m\otimes\C^n$ for the two different copies of $\C^n$.
In the next part we make use also of those two flattenings and show how to construct more polynomials in $I_{n,m}$.
This new construction does not require the integer $m$ to be as large.
\end{rem}

Instead of fixing $m$ immediately and working in $\Sing_{m,n}$ we consider more generally a tensor $X\in\C^m\otimes\C^n\otimes\C^n$ and study its image by the flattening map in $\C^n\otimes\C^n$.

\begin{prop}
\label{prop:spaces_almost_maximal}
Let $k$ be a positive integer and let $L\subset\C^n\otimes\C^n$ be a subspace of dimension $k$. If $L$ is contained in $D_n$ and $k\ge n^2-2n+2$, then after a change of coordinates either every matrix in $L$ has
\begin{enumerate}
\item a zero-row, or
\item a zero-column, or
\item the form
\[
\begin{pmatrix}
\lambda_1 	& \dots & \lambda_{n^2-2n+1} & \lambda_{n^2-2n+2}\\
\vdots		& \dots &	0	& 0\\
\vdots		& \dots &	\vdots	& \vdots\\
\lambda_n	& \dots & 0 & 0
\end{pmatrix},
\]
or
\item the form 
\[
\begin{pmatrix}
\lambda_1 	& \dots & \lambda_{n^2-2n+1} & \lambda_{n^2-2n+2}\\
\vdots		& \dots &	0	& 0\\
\vdots		& \dots &	\vdots	& \vdots\\
\lambda_n	& \dots & 0 & 0
\end{pmatrix}^T.
\]
\end{enumerate}
\end{prop}
\begin{proof}
By \cite[Corollary 5.1]{CI} the Fano scheme $\mathbf{F}_{k-1}(D_n)$ consists only of compression subspaces, i.e. there exists $0\le s\le n-1$ such that $L$ is contained in the compression space $\mathfrak{C}_{k-1}(s)$. By definition this means there exists a subspace $V\subset \C^n$ of dimension $s+1$ that is mapped to a subspace $W\subset\C^n$ of dimension $s$ by every element of $L$. 

By \cite[Remark 5.2]{CI} the only possiblities for such $s$ are $s=0,1,n-1,n-2$ if $k\ge n^2-2n+2$. These correspond exactly to the four cases in the claim.
\end{proof}

\begin{prop}
\label{prop:more_equations}
Let $X\in\Sing_{m,n}$. Then the map $\C^m\to\C^n\otimes\C^n$ has rank at most $n^2-2n+2$ or one of the two flattenings $\C^n\to\C^m\otimes\C^n$ has rank at most $n-1$.
\end{prop}
\begin{proof}
Let $L$ be the image of the map $\C^m\to\C^n\otimes\C^n$ and denote its dimension by $k$. If $k>n^2-2n+2$, then by \cref{prop:spaces_almost_maximal} the subspace $L$ has one of the two given forms (1), (2) in Proposition \ref{prop:spaces_almost_maximal}. Hence, the rank of one of the maps $\C^n\to\C^m\otimes\C^n$ is at most $n-1$ depending if we are in case (1) or (2).
\end{proof}

\begin{rem}
From \cref{prop:more_equations} we get more equations for $\Sing_{n,m}$ as follows. Let $J_M$ be the ideal generated by the $n^2-2n+3$ minors of the map $\C^m\to\C^n\otimes\C^n$ and let $J_1,J_2$ be the ideals generated by the $n$ minors of the other two flattenings. Then $J_M\cdot J_1\cdot J_2\subset (I_{n,m})_{n^2+3}$.

Compared to \cref{thm:fano_empty} we should consider $m\ge n^2-2n+3$ so that the map $\C^m\to\C^n\otimes\C^n$ can have large enough rank.

Using this we get equations in the case $n=3$ and $m=6$ whereas earlier we needed $m\ge 7$ in this case.
\end{rem}

\begin{rem}
Our theorems suggest the following idea to construct more equations for $\Sing_{m,n}$. Consider a $k$-dimensional subspace $L\subset\C^n\otimes\C^n$ and study possible forms of this subspace using the Fano scheme $\mathbf{F}_{k-1}(D_n)$. The conditions/equations of the Fano scheme translate to equations for $\Sing_{m,n}$ as above.
\end{rem}